\documentclass[reqno]{amsart}
\usepackage{amsmath}
\def\eps{\varepsilon}

\newtheorem{lemma}{Lemma}
\newtheorem{theorem}[lemma]{Theorem}
\newtheorem{corollary}[lemma]{Corollary}
\newtheorem{proposition}[lemma]{Proposition}
\theoremstyle{definition}

\theoremstyle{remark}

\usepackage{hyperref}

\newcommand{\sgn}{\operatorname{sgn}}

\begin{document}

\title[Renormalization of active scalar equations]%
{Renormalization of active scalar equations}
%\date{\today}
\author[I. Akramov]{Ibrokhimbek Akramov}
\address[I. Akramov]{Institut f\"{u}r Angewandte Mathematik \\
    Leibniz Universit\"{a}t Hannover \\
    30167 Hannover \\
    Germany}
\email{akramov@ifam.uni-hannover.de}

\author[E. Wiedemann]{Emil Wiedemann}
\address[E. Wiedemann]%
{Institut f\"{u}r Angewandte Mathematik \\
    Leibniz Universit\"{a}t Hannover \\
    30167 Hannover \\
    Germany}
\email{wiedemann@ifam.uni-hannover.de}

\begin{abstract}
We consider transport equations with an incompressible transporting vector field. Whereas smooth solutions of such equations conserve every $L^p$ norm simply by the chain rule, the question arises how regular a weak solution needs to be to guarantee this conservation property. Whereas the classical DiPerna-Lions theory gives sufficient conditions in terms of the regularity of the coefficients, with no regularity requirement on the transported scalar, we give here sufficient conditions in terms of the combined regularities of the coefficients and the scalar. This is motivated by the case of active scalar equations, where the transporting vector field has the same regularity as the transported scalar. We use commutator estimates similar to those of Constantin-E-Titi in the context of Onsager's conjecture, but we require novel arguments to handle the case of $L^p$ norms when $p\neq 2$.   
\end{abstract}

\maketitle

\section{Introduction}
Let $u:[0,T]\times\mathbb{T}^d\to\mathbb{R}^d$ be a divergence-free vector field. The transport equation
\begin{equation}\label{transport}
\partial_t\theta+u\cdot\nabla\theta=0
\end{equation}
is a fundamental ingredient in many mathematical models in fluid mechanics, for instance in the two-dimensional Euler equations (where the vorticity is transported by the flow), the semiquasigeostrophic (SQG) and incompressible porous media (IPM) equations, which belong to the class of active scalar equations, or (when the velocity is not necessarily divergence-free) in compressible fluid dynamics, where the continuity equation represents the conservation of mass.

It is convenient to assume in addition
\begin{equation}\label{int}
\int_{\mathbb{T}^d}udx=0, \quad\int_{\mathbb{T}^d}\theta dx=0,
\end{equation}
which allows us to work in homogeneous function spaces.

The transport equation comes, at least formally, with infinitely many conserved quantities: Indeed, if $\beta:\mathbb{R}\to\mathbb{R}$ is any smooth function, then multiplication of~\eqref{transport} with $\beta'(\theta)$ and application of the chain rule yields
\begin{equation}\label{renormalized}
\partial_t\beta(\theta)+u\cdot\nabla\beta(\theta)=0.
\end{equation} 
Passing from the original equation~\eqref{transport} to equation~\eqref{renormalized} has been known (since the seminal work~\cite{dipernalions}) as \emph{renormalization}. It is an essential technique e.g.\ in the analysis of the Boltzmann equation~\cite{dipernalionsboltzmann} and in the theory of weak solutions of the compressible Navier-Stokes equations~\cite{lionscomp, feireislcompact}.

When $\theta$ is not Lipschitz, however, the use of the chain rule is no longer justified, and many counterexamples to renormalization have been discovered~\cite{aizenman, depauw, colluorau, abc1, abc2, cgsw1, cgsw2}.

The situation is comparable for other classes of partial differential equations: For scalar conservation laws of the form
\begin{equation*}
\partial_tu+\operatorname{div}_xF(u)=0,
\end{equation*} 
if $\beta$ is a smooth function, then multiplication with $\beta'(u)$ formally gives the companion law
\begin{equation}\label{entropy}
\partial_t\beta(u)+\operatorname{div}_xq(u)=0,
\end{equation} 
for $q$ satisfying $q'=\beta'F'$. If $\beta$ is convex, it is called an entropy, and $q$ is the corresponding entropy flux. Also for this class of equations, the existence of nonsmooth weak solutions that fail to satisfy~\eqref{entropy} is classically known (shocks).

Systems of equations typically also admit companion laws, but not an infinite number of them. For instance, the incompressible Euler equations formally satisfy the local energy equality
\begin{equation*}
\partial_t\frac{|u|^2}{2}+\operatorname{div}_x\left[\left(\frac{|u|^2}{2}+p\right)u\right]=0,
\end{equation*} 
and this equality may be violated for irregular weak solutions~\cite{scheffer, shnirelman}.

It has therefore been of great interest to provide sufficient conditions, in terms of the regularity of solutions, that ensure the conservation of formally conserved quantities in each of the mentioned cases. One can broadly distinguish two different approaches to this problem: The theory of DiPerna-Lions for transport and continuity equations gives a Sobolev regularity condition on the transporting field $u$ such that every solution $\theta$ of~\eqref{transport} is renormalized, even when $\theta$ itself has no regularity at all. Results of this type are extremely useful in cases where there are a priori estimates on $\nabla u$, but not on $\nabla\theta$, as it occurs e.g.\ for the compressible Navier-Stokes system, or the two-dimensional Euler equations~\cite{eyink01, lopes06} with vorticity in $L^p$.

On the other hand, in the presence of advection, the solution is transported (vaguely speaking) by itself, or by a field with similar regularity as itself. This is the case, for instance, for Euler and Navier-Stokes models, and for active scalar equations. Then one should evenly split the required regularity between the quantities appearing in the advective term. In the case of the incompressible Euler equations, this leads to the famous Onsager exponent 1/3, and the prototypical result of this kind is in the work of Constantin-E-Titi~\cite{constetiti} (see also Eyink~\cite{eyink94}). There it is shown that weak solutions of the Euler equations conserve energy provided they possess fractional Besov differentiability of order greater than 1/3.

The method of Constantin-E-Titi has been refined and extended. In particular, the Besov condition has been optimized~\cite{duchonrobert, CCFS08, FW17}, density-dependent systems like the compressible Euler and Navier-Stokes systems have been considered~\cite{LSh, FGGW, CY17, DE17}, general systems of conservation laws have been treated~\cite{Martin}, and boundary effects have been taken into account~\cite{RRS16, RRS18, BT18, BTW18, DN18}. 

Here, we consider the problem of renormalization of the transport equation~\eqref{transport} with $\beta=|\cdot|^p$, and we prove: 
\begin{theorem}\label{mainth}
Let $u$ be weakly divergence-free and $\theta$ a weak solution of~\eqref{transport}. Suppose $\theta\in L^{p_1}(0, T; \dot{B}_p^{\alpha, \infty}(\mathbb{T}^d))$ and $u\in L^{q_1}(0, T; \dot{B}_q^{\beta, \infty}(\mathbb{T}^d))$, where
\begin{itemize}
\item $2\alpha+\beta>1$;
\item $\frac2p+\frac1q=1$;
\item $2<p_1\le \frac{pd}{(d-p\alpha)_+}$,\quad $1\leq q_1\le \frac{qd}{(d-q\beta)_+}$;
\item $\frac{2}{p_1}+\frac1{q_1}< 1$.
\end{itemize}  
Then, we have
\begin{equation*}
\partial_t|\theta|^r+u\cdot\nabla |\theta|^r=0
\end{equation*}
for every $r\geq 1$ for which $|\theta|^ru$ is locally integrable.
\end{theorem}
See the next section for definitions of the relevant function spaces. A weak solution of~\eqref{transport} is defined, as usual, in the sense of districutions, i.e.\ we require 
\begin{equation*}
\int_0^T\int_{\mathbb{T}^d}(\partial_t\varphi+u\cdot\nabla\varphi)\theta dxdt=0
\end{equation*}
for every test function $\phi\in C_c^\infty((0,T)\times\mathbb{T}^d)$. The maximal value of $r$ can be explicity determined in terms of $p,p_1,q,q_1,\alpha,\beta$, see Lemma~\eqref{analyticity1} below.

The main new difficulty compared to the mentioned previous works consists of the fact that the term $|\theta|^{p-2}$, which will now appear in the commutator, is possibly unbounded when $\theta$ assumes values close to zero (when $p<2$) or close to $\infty$ (when $p>2$). This difficulty corresponds to the observation that the function $|\cdot|^p$ is not $C^2$ in the range of $\theta$, and this is precisely the reason why our case is not covered by the general theory of Gwiazda et al.~\cite{Martin}.

We specialise to the case of an active scalar equation. Such an active scalar equation is~\eqref{transport} together with a nonlocal coupling
\begin{equation}\label{coupling}
 u=\mathcal{T}[\theta];
\end{equation} 
Here the operator $\mathcal{T}$ is represented in frequency space by a Fourier multiplier:
\begin{equation*}
\hat{u}(\xi)=\widehat{\mathcal{T}[\theta]}(\xi)=m(\xi)\widehat{\theta}(\xi),
\quad \xi\in \mathbb{Z}^d\backslash\{0\}.
\end{equation*}
We assume that $m$ is an $L^p$ multiplier (see \cite{GRA1,GRA2}). This class of equations includes, in particular, the SQG and IPM equations. Then we have:
\begin{theorem}\label{mainthactive}
Let $\theta$ be a weak solution of the active scalar equation~\eqref{transport} together with the coupling~\eqref{coupling}.
Assume that $m(\xi)$ is an $L^3$ multiplier and $\theta\in  L^{p_1}([0, T], \dot{B}_3^{\alpha, \infty}(\mathbb{T}^d))$
for $\frac{1}{3}<\alpha<1$ and $p_1\le\frac{3d}{d-3\alpha}$.
Then, we have
\begin{equation*}
\partial_t|\theta|^r+\mathcal{T}[\theta]\cdot\nabla |\theta|^r=0
\end{equation*}
for every $r\geq 1$ for which $\theta$ is locally in $L^{r+1}((0,T);\mathbb{T}^d)$.
\end{theorem} 

It is known that, at low regularity, the method of convex integration is applicable to certain active scalar equations, and yields non-renormalized weak solutions; see, for instance,~\cite{CFG, Sz12, RSh, PI, BSV}. The expected optimal exponent 1/3 has not yet been reached in the construction of non-renormalized solutions.

As a final remark, let us point out that an alternative (and arguably slightly simpler) proof than the one given here would proceed by an approximation argument, replacing $|\cdot|^p$ by a uniformly $C^2$ function. The advantage of our proof, however, is that it mostly does not depend on the existence of infinitely many conserved quantities (except of course for the analyticity argument), and is thus transferable to other problems. In particular, we are currently applying the techniques of this paper to the compressible Euler equations with possible vacuum (in progress).

\section{Preliminaries}
We derive some estimates in Besov spaces, which will be crucial later on. Throughout this note we denote by $\mathbb{T}^d$ the $d$-dimensional torus, i.e.\
the cube $[-\pi, \pi]^d$ with the opposite edges identified:
$\mathbb{T}^d:=\mathbb{R}^d/(2\pi \mathbb{Z})^d$ and functions
defined on the torus will thus be identified with periodic functions
defined on $\mathbb{R}^d$. Next, we introduce a standard mollifier and its rescalings: Let $\eta$ be a smooth function defined on
$\mathbb{R}^d$ satisfying the conditions $\eta \in
C^\infty(\mathbb{R}^d)$, $\eta\geq 0$,
\begin{equation}{\nonumber}
\eta(x)=\begin{cases} c\quad \mbox{for} \quad
|x|\le \frac12,\\
0\quad\mbox{for}\quad |x|\ge 1,
\end{cases}
\end{equation}
and
$$
\int_{\mathbb{R}^d}\eta(x)dx=1.
$$
Furthermore, define
$$
\tilde\eta^{\eps}(x):=\frac1{\eps^d}\eta\left(\frac{x}{\eps}\right),
$$
then clearly
$$
\int_{\mathbb{R}^d}\tilde\eta^{\eps}(x)dx=1
$$
and $\operatorname{supp} (\tilde\eta^{\eps})\subset \overline{B_\epsilon(0)}= \{|x|\le \eps\}$. We
denote by $\eta^{\eps}$ the natural periodic continuation of the
function $\tilde\eta^{\eps}$ to $\mathbb{R}^d$. So,
$\eta^{\eps}$ can be considered as it is defined on
$\mathbb{T}^d$. Then for an integrable function $f$ defined on
$\Omega:=[0, T]\times \mathbb{T}^d$ and for $(t, x)\in \Omega$ we
denote the space mollification
$$
f^\eps(t, x):=(\eta^\epsilon*f)(t,x):=\int_{B_\epsilon(0)}\eta^{\eps}(y)f(t, x-y) dy.
$$
Obviously, $f^\eps$ is a space-periodic function defined on $[0, T]\times
\mathbb{R}^d$. Thus, it can be considered as a function on
$\Omega$. Moreover, $f^\eps(t, \cdot)\in C^\infty(\mathbb{T}^d)$ for a.e. $t\in [0, T]$.  The following
relations are obvious and well-known:
\begin{align}
&\theta^{\eps}=\eta^{\eps}\ast\theta,\\&
\mathcal{T}[\theta^{\eps}]=\mathcal{T}[\theta]^{\eps},\\& (\partial u)^{\eps}=\partial
u^{\eps},\\& (\mbox{div}_x u)^{\eps}=\mbox{div}_x u^{\eps}.
\end{align}
We briefly recall some properties of Besov spaces and their equivalent definition in terms of the Littlewood-Paley decomposition. First, by assumption~\eqref{int}, for our purposes it suffices to work in the homogeneous Besov space $\dot{B}_p^{\alpha,\infty}(\mathbb{T}^d)$ ($0< \alpha<1$, $1\leq p\leq\infty$), which consists of those functions $f\in L^p(\mathbb{T}^d)$ that have mean zero, and for which the norm
$$
\|f\|_{\dot{B}_p^{\alpha,\infty}(\mathbb{T}^d)}=\sup_{y\in
    \mathbb{T}^d}
\frac{\|f(\cdot+y)-f(\cdot)\|_{L^p(\mathbb{T}^d)}}{|\xi|^\alpha}
$$ is finite.
An equivalent definition of the Besov space can be given in Fourier space in the following way:
  Let $\omega:\mathbb{R}^d\to\mathbb{R}$, $0\leq\omega\leq 1$, be a smooth function such that
\begin{equation}{\nonumber}
    \omega(\xi)=\begin{cases} 1\quad \mbox{for} \quad
        |\xi|\le \frac12 \\
        0\quad\mbox{for}\quad |\xi|\ge1.
    \end{cases}
\end{equation}
We consider a partition of unity given by $\omega$ and
$\varphi(\xi)=\omega(\xi/2)-\omega(\xi)$:
$$
\omega(\xi)+\sum_{\nu=0}^\infty \varphi(\lambda_\nu^{-1}\xi)=1,
$$
where $\lambda_\nu:=2^\nu$. The Littlewood-Paley decomposition of a function $f$ with
 zero mean is then given as $f=\sum_0^\infty\Delta_k f$, where
$$
%\Delta_{-1}
%f:=\mathcal{F}^{-1}(\omega \hat f), \quad
 \Delta_{k}
f:=\mathcal{F}^{-1}\left(\varphi\left(\lambda_k^{-1}\cdot\right)
\hat f\right),
$$
with the Fourier transform
$$
\hat f(\xi):=\int_{\mathbb{T}^d} e^{-2\pi i(\xi, x)} f(x)dx
$$
and the inverse Fourier transform %\footnote{check if a constant factor $2\pi$ is missing somewhere}
$$
\mathcal{F}^{-1}(\hat f)(x):=\sum_{\xi\in\mathbb{Z}^d\setminus\{0\}} e^{2\pi i(\xi, x)}
\hat f(\xi)=f(x).
$$
The Besov norm can then be equivalently characterized as %\footnote{we only need $q=\infty$}
\begin{equation*}
\|f\|_{\dot{B}_p^{\alpha,\infty}}=\sup_{k\geq0}2^{k\alpha
}\|\Delta_{k}f\|_{L^p}
\end{equation*}
for $\alpha\in \mathbb{R}$, $1\le p\leq\infty$.

Let us recall the embedding theorem for Besov spaces from
\cite{Bahori}, Proposition 2.20, p.~64:
\begin{proposition}\label{embedding}
Let $1\leq p_1\leq p_2\leq \infty$.
Then, the space $\dot{B}^{\alpha, \infty}_{p_1}$ is continuously embedded in
$\dot{B}^{\alpha-d\left(\frac{1}{p_1}-\frac{1}{p_2}\right), \infty}_{p_2}$.%\footnote{we need to make sure that $\alpha-d\left(\frac{1}{p_1}-\frac{1}{p_2}\right)>0$. This also pertains to Lemma 4.}
\end{proposition}

\begin{lemma}\label{L1}
 Let  $\theta\in L^p(\mathbb{T}^d)$ and $p\ge s-1\ge1$,  then the inequality
\begin{equation}\label{thetarealpart}
\left\||\theta^{\eps}|^{s-1}\right\|_{L^p(\mathbb{T}^d)}\leq
C(q)\eps^{-\frac{d(s-2)}{p}}\|\theta\|_{L^p(\mathbb{T}^d)}^{s-1}
\end{equation}
holds, where $q:=\frac{p(s-1)}{2-s-p+ps}.$
\end{lemma}
\begin{proof}  $\theta^\eps(\cdot)$ is a smooth function defined on $\mathbb{T}^d$.  We consider an estimate for the
norm of $\theta^\eps(\cdot)$ by employing Young's inequality
\begin{equation}\label{thetanorm}
\begin{split}
\left\||\theta^{\eps}(\cdot)|^{s-1}\right\|_{L^{p}(\mathbb{T}^d)}&=\left(\int_{\mathbb{T}^d}|\theta^{\eps}(
x)|^{p(s-1)}dx\right)^{\frac{1}{p}}\\
&=\|\theta^{\eps}(\cdot)\|_{L^{p(s-1)}(\mathbb{T}^d)}^{s-1}\leq\|\theta(\cdot)\|_{L^{p}(\mathbb{T}^d)}^{s-1}\left\|\eta^{\eps}\right\|_{L^{q}(\mathbb{R}^d)}^{s-1},
\end{split}
\end{equation}
where $\frac{1}{p(s-1)}=\frac{1}{p}+\frac{1}{q}-1.$ But we have
\begin{equation*}
\begin{split}
\|\eta^{\eps}\|_{L^{q}(\mathbb{R}^d)}&=\left(\int_{B_\epsilon(0)}\left|\frac1{\eps^d}\eta\left(\frac{x}{\eps}\right)\right|^{q}dx\right)^{\frac{1}{q}}\\&
=\frac{1}{\eps^d}\left(\int_{B_1(0)}\eps^d\eta\left(x_1\right)^{q}dx_1\right)^{\frac{1}{q}}\\&
=C(q)\eps^{d\left(\frac{1}{q}-1\right)}=C(q)\eps^{-\frac{d(s-2)}{p(s-1)}}
\end{split}
\end{equation*}
with
$$
C(q):=\left(\int_{B_1(0)}\eta(x_1)^{q}dx_1\right)^{\frac{1}{q}}.
$$
In view of the above results, we achieve the desired expression.
\end{proof}
\begin{lemma}\label{estfornab}
    For any $s\ge2$, $p>2$, $0<\alpha<1$, and $\theta\in \dot{B}_p^{\alpha, \infty}(\mathbb{T}^d)$ the following inequality holds:
    \begin{equation}\label{estnabth}
    \|\nabla_x\theta^{\eps}\|_{L^{p(s-1)}(\mathbb{T}^d)}\leq C\eps^{\alpha
        -\frac{d(s-2)}{p(s-1)}-1}\|\theta\|_{\dot{B}_p^{\alpha, \infty}(\mathbb{T}^d)}.
    \end{equation}

\end{lemma}
\begin{proof} We use the following well-known inequality for any $0<\beta<1$ and $1\leq q\leq \infty$, cf. \cite{constetiti, FGGW}:
\begin{equation*}
\left\|\nabla_x\theta^{\eps}\right\|_{L^{q}(\mathbb{T}^d)}\leq C \eps^{\beta-1}\|\theta\|_{\dot{B}_{q}^{\beta, \infty}(\mathbb{T}^d)},
\end{equation*}
so that the choice $q=p(s-1)$, $\beta=\alpha
-\frac{d(s-2)}{p(s-1)}$ yields

\begin{equation}\label{em01}
\left\|\nabla_x\theta^{\eps}\right\|_{L^{p(s-1)}(\mathbb{T}^d)}\leq C \eps^{\alpha
-\frac{d(s-2)}{p(s-1)}-1}\|\theta\|_{\dot{B}_{p(s-1)}^{\alpha
-\frac{d(s-2)}{p(s-1)}, \infty}(\mathbb{T}^d)}.
\end{equation}
Then by Proposition \ref{embedding}, we have
\begin{equation}\nonumber
\|\theta\|_{\dot{B}_{p(s-1)}^{\alpha-d\left(\frac{1}{p}-\frac{1}{p(s-1)}\right),\infty}(\mathbb{T}^d)}\leq C \|\theta\|_{\dot{B}_p^{\alpha, \infty}(\mathbb{T}^d)}.
\end{equation}
This implies
\begin{equation}\label{em02}
\|\theta\|_{\dot{B}_{p(s-1)}^{\alpha
        -\frac{d(s-2)}{p(s-1)},\infty}(\mathbb{T}^d)}\leq C \|\theta\|_{\dot{B}_p^{\alpha, \infty}(\mathbb{T}^d)}.
\end{equation}
Combining \eqref{em01} and \eqref{em02}, we achieve the desired inequality \eqref{estnabth}.
\end{proof}
Integrating in time, we obtain
\begin{corollary}
For any $s\ge2$, $1\leq p_2\leq\infty$, $p_1>2$, $0<\alpha<1$, and  $\theta\in
L^{p_2}(0, T; \dot{B}_{p_1}^{\alpha, \infty}(\mathbb{T}^d))$, the following inequality
\begin{equation}\label{estnabth2}
\|\nabla_x\theta^{\eps}\|_{L^{p_2}(0, T;\,L^{p_1(s-1)}(\mathbb{T}^d))}\leq C\eps^{\alpha
-\frac{d(s-2)}{p_1(s-1)}-1}\|\theta\|_{L^{p_2}(0, T;\,\dot{B}_{p_1}^{\alpha, \infty}(\mathbb{T}^d))}
\end{equation}
holds.
\end{corollary}

Recall that a Fourier multiplier operator acts on a function $f:\mathbb{T}^d\to\mathbb{R}$ via
\begin{equation*}
\mathcal{T}f=\mathcal{F}^{-1}(m\hat f),
\end{equation*}
where $m:\mathbb{Z}^d\to\mathbb{C}$ is the symbol. The symbol is called an $L^p$-multiplier if $\mathcal{T}$ is a bounded operator from $L^p$ to itself. From the Fourier characterization of Besov spaces, one can see that in this case $\mathcal{T}$ is also bounded from a Besov space into itself. More precisely:

\begin{proposition}\label{boundedness}
Let $\mathcal{T}: L^p(\mathbb{T}^d)\to L^p(\mathbb{T}^d)$ be an operator with symbol $m(\xi)$. If
$m$ is an $L^p$- multiplier then $\mathcal{T}: \dot{B}_p^{\alpha,\infty}(\mathbb{T}^d)\to
\dot{B}_p^{\alpha,\infty}(\mathbb{T}^d)$ is a bounded operator for $0<\alpha<1$ and $1<p<\infty$.
\end{proposition}
\begin{proof} Since $m$ is an $L^p$ multiplier,
\begin{equation*}
\|\mathcal{T}(2^{k\alpha}\Delta_k f)\|_{L^p}\leq C\|2^{k\alpha}\Delta_k
f\|_{L^p}.
\end{equation*}
Therefore, we obtain
\begin{equation*}
\sup_k\left\|\mathcal{T}(2^{k\alpha}\Delta_k
f)\right\|_{L^p}\leq C
\sup_k\left\|2^{k\alpha}\Delta_k
f\right\|_{L^p}
\end{equation*}
and thus
\begin{equation}\label{eq*}
\sup_k\left\|2^{k\alpha}\Delta_k
\mathcal{T}f\right\|_{L^p}\leq C
\sup_k\left\|2^{k\alpha}\Delta_k
f\right\|_{L^p},
\end{equation}
as $\mathcal{T}$ is linear and commutes with $\Delta_k$.
Hence we obtain
\begin{equation*}
\left\|\mathcal{T}f\right\|_{\dot{B}_p^{\alpha, \infty}}\leq
C\|f\|_{\dot{B}_p^{\alpha,\infty}}.
\end{equation*}
\end{proof}

\section{Analyticity}
Here we will prove a Lemma which shows that the integral appearing in the
weak formulation of the renormalized transport equation is an analytic function with respect to the renormalization exponent. For a complex number $z\in\mathbb{C}$, we write $\Re z$ for its real part. We also use the notation
\begin{equation*}
x_+:=
\begin{cases}
0, & \text{if } x\le 0,\\
x, &\text{if } x>0.
\end{cases}
\end{equation*}

\begin{lemma}\label{analyticity1}
If $\theta\in L_{loc}^{p_1}((0, T), L^p(\mathbb{T}^d))$,  with $p\ge p_1$ and
$u\in L_{loc}^{q_1}((0, T), L^q(\mathbb{T}^d))$,  with $q\ge q_1$,  $p_1>1$ and $\frac1{p_1}+\frac1{q_1}\leq1$,
then for any $\varphi\in C_c^\infty((0, T)\times\mathbb{T}^d)$
the function
\begin{equation}\nonumber
F(z)=-\int_{0}^{T}\int_{\mathbb{T}^d}
|\theta|^{z}\left(\partial_t\varphi+u\cdot\nabla\varphi\right) dxdt
\end{equation}  is an analytic function on the strip $0<\Re z<\frac{p_1}{q_1'},$
where $q_1'$ is the dual exponent to $q_1$ e.g.
$\frac1{q_1}+\frac1{q_1'}=1$. In particular, if  $\theta\in L_{loc}^{p_1}(0, T; \dot{B}_p^{\alpha, \infty}(\mathbb{T}^d))$,
$u\in L_{loc}^{q_1}(0, T; \dot{B}_q^{\beta, \infty}(\mathbb{T}^d))$,  with
$p_1\leq\frac{pd}{(d-p\alpha)_+}$, $q_1\leq\frac{qd}{(d-q\beta)_+}$, $\alpha>0$, and $\frac2{p_1}+\frac1{q_1}<1$,
then there exists a positive number $\gamma>0$ such that  $F$ is an analytic function on the strip
$0<\Re(z)<2+\gamma$.
\end{lemma}

\begin{proof} %Without loss of generality, we can assume that $p_1=p$ and $q_1=q$.
We want to show that
\begin{equation*}
F'(z)=\int_{0}^{T}\int_{\mathbb{T}^d}|\theta(t, x)|^z\log|\theta(t, x)|\left(\partial_t\varphi+u\cdot\nabla\varphi\right)dxdt
\end{equation*} is a well-defined function.
Note that for any fixed $\delta>0$, there exists $C(\delta)>0$ such
that the inequality
\begin{equation*}
|\log |y||\leq C(\delta)(1+|y|)^{2\delta}|y|^{-\delta}
\end{equation*}
holds for any $y\in \mathbb{R}\setminus\{0\}$.
We have
\begin{equation}\label{derest}
|F'(z)|\leq C\int_{0}^{T}\int_{\mathbb{T}^d}|\theta(t, x)|^{\Re z}\left(1+|\theta(t, x)|\right)^{2\delta}|\theta(t, x)|^{-\delta}|\partial_t\varphi+u\cdot\nabla\varphi|dxdt.
\end{equation}
For $|\theta|\geq 1$, we have
$$
|\theta|^{\Re z-\delta}\left(1+|\theta|\right)^{2\delta}\leq 2^{2\delta}|\theta|^{\Re z+\delta},
$$
while, for $|\theta|\leq 1$, we obtain
$$
|\theta|^{\Re z-\delta}\left(1+|\theta|\right)^{2\delta}\leq 2^{2\delta}|\theta|^{\Re z-\delta},
$$
further, let $$z\in\{\Delta\leq\Re z\leq\frac{p_1}{q_1'}-\Delta\}\quad\text{for sufficiently small}\quad \Delta:=2\delta.$$
Using the above we have
\begin{equation*}
\begin{split}
|F'(z)|&\leq C\int_{0}^{T}\int_{\mathbb{T}^d}
|\theta(t, x)|^{\Re z}\left(1+|\theta(t, x)|\right)^{2\delta}|\theta(t, x)|^{-\delta}|\partial_t\varphi+u\cdot\nabla\varphi
|dxdt\\& \leq
C\int_{[0, T] \times\mathbb{T}^d\cap\{|\theta|\ge1\}}|\theta(t,
x)|^{\Re z+\delta}|\partial_t\varphi+u\cdot\nabla\varphi
    |dxdt\\&+C\int_{[0, T] \times\mathbb{T}^d\cap\{|\theta|\le1\}}|\theta(t,
x)|^{\Re z-\delta}|\partial_t\varphi+u\cdot\nabla\varphi
    |dxdt\\& \leq
C\int_{[0, T] \times\mathbb{T}^d\cap\{|\theta|\ge1\}}|\theta(t,
x)|^{\frac{p_1}{q_1'}-\Delta+\delta}|\partial_t\varphi+u\cdot\nabla\varphi
    |dxdt\\&+C\int_{[0, T] \times\mathbb{T}^d\cap\{|\theta|\le1\}}|\theta(t,
x)|^{\Delta-\delta}|\partial_t\varphi+u\cdot\nabla\varphi
    |dxdt\\& \leq
    C\int_{[0, T] \times\mathbb{T}^d}(|\theta(t,
x)|^{\delta}+|\theta(t,
x)|^{\frac{p_1}{q_1'}-\delta})|\partial_t\varphi+u\cdot\nabla\varphi
    |dxdt <\infty
\end{split}
\end{equation*}
and note that the last integral converges. Thus the integral for $F'(z)$ converges
uniformly in  $\{z\in \mathbb{C}: \Delta\leq\Re z\leq
\frac{p_1}{q_1'}-\Delta\}$ and so  there exists $F'(z)$ for any $z\in
\{0<\Re z<\frac{p_1}{q_1'}\}$. Consequently, $F(z)$ is an analytic
function on the strip $0<\Re z<\frac{p_1}{q_1'}$.

For the ``in particular" part of the lemma, observe that by Besov embedding (Proposition~\ref{embedding}), $\dot{B}_{p,\infty}^\alpha$ embeds continuously into $L^{r}$ with $r:=\frac{pd}{(d-p\alpha)_+}$, and likewise  $\dot{B}_{q,\infty}^\beta$ embeds continuously into $L^{s}$ with $s:=\frac{qd}{(d-q\beta)_+}$ (with the convention $r=\infty$ if the denominator is zero); by assumption, these exponents are not smaller than $p_1$ and $q_1$, respectively, and so the first part of the Lemma applies. Since 
\begin{equation*}
\frac{2}{p_1}+\frac{1}{q_1}<1,
\end{equation*} 
we obtain $\frac{p_1}{q_1'}>2$, and the conclusion follows.

\end{proof}

\section{Commutator estimates}\label{Ons}
%In this section, we derive the theorem, stated below, which will be used later to prove our main result, in Theorem \ref{mainth}.

\begin{theorem}\label{Thmain}
    Let $(\theta, u)$
    be a weak solution of \eqref{transport} on $(0, T)\times\mathbb{T}^d$. Assume that
    $\theta\in L^{p_1}(0, T; \dot{B}_p^{\alpha, \infty}(\mathbb{T}^d)),\quad u\in L^{p_2}(0, T; \dot{B}_q^{\beta, \infty}(\mathbb{T}^d))$
    with $2<p_1\le \frac{pd}{(d-p\alpha)_+}$, $p_2\le \frac{qd}{(d-q\beta)_+}$,
    $\frac{2}{p}+\frac{1}{q}=1$ and $\frac{2}{p_1}+\frac1{p_2}< 1$, where $1\leq p,q\leq\infty$ and $0<\alpha, \beta<1$. Let $\gamma>0$ and $2\alpha+\beta>1+\frac{d\gamma}{p}$. Then,
    for any $z\in S:=\{z\in\mathbb{C}:2<\Re z<2+\gamma\}$, the relation
    \begin{equation}\label{Th1}
    \partial_t|\theta|^z+\operatorname{div}_x(|\theta|^zu)=0
    \end{equation} holds in the sense of distributions on $(0, T)\times\mathbb{T}^d$.
   % In particular, for any real $p\in S$ the conservation of $L^p$ norm of $\theta$ holds.
    \end{theorem}
\begin{proof}

By mollifying \eqref{transport} with respect to spatial variables, we obtain
\begin{equation}\label{(14)}
\partial_t\theta^{\eps}+\mbox{div}_x(\theta u)^{\eps}=0.
\end{equation}
If $\theta^{\eps}$ is a weak solution to equation \eqref{(14)}
then it has the weak derivative with respect to $t$. Moreover, $\partial_t\theta^{\eps}$ belongs to
$L^\frac{p_1p_2}{p_1+p_2}([0, T]\times\mathbb{T}^d)$ under the conditions of Theorem~\ref{Thmain}.

Note that if $\Re z>2$, then  $|\theta^{\eps}(t,\cdot)|^{z-1}\sgn\theta^{\eps}(t, \cdot)$ is  a continuously
differentiable function on $\mathbb{T}^d$ for a.e. $t\in (0, T)$.

Let $\varphi\in C_c^{\infty}\left((0, T)\times \mathbb{T}^d\right)$.
The multiplication of $\partial_t\theta^{\eps}$ with
$z|\theta^{\eps}|^{z-1} \sgn\theta^{\eps}\varphi$ is well-defined in the sense of distributions, owing to the
condition $\frac{2+\delta}{p_1}+\frac1{p_2}\le1$ for some positive
number $\delta>0$. Thus multiplication of \eqref{(14)} with
$z|\theta^{\eps}|^{z-1} \sgn\theta^{\eps}\varphi$ and integration
over time and space gives
\begin{equation*}
\int_{0}^{T}\int_{\mathbb{T}^d}\partial_t\theta^{\eps}
z|\theta^{\eps}|^{z-1}\sgn\theta^{\eps}\varphi
dxdt+\int_{0}^{T}\int_{\mathbb{T}^d}\mbox{div}_x(\theta
u)^{\eps}z|\theta^{\eps}|^{z-1}\sgn\theta^{\eps}\varphi dxdt=0.
\end{equation*}
Here we take $\eps>0$ small enough so that $\varepsilon_1\ge \varepsilon$, where $\varepsilon_1$ is chosen such that $\operatorname{supp} \varphi\subset
(\varepsilon_1, T-\varepsilon_1)\times\mathbb{T}^d$.
%\begin{equation}
%\int_{\eps_1}^{T-\eps_1}\int_{\mathbb{T}^d}\partial_t\theta^{\eps}
%z|\theta^{\eps}|^{z-1}sgn\theta^{\eps}\varphi
%dxdt+\int_{\eps_1}^{T-\eps_1}\int_{\mathbb{T}^d}\mbox{div}_x(\theta
%u)^{\eps}z|\theta^{\eps}|^{z-1}sgn\theta^{\eps}\varphi dxdt=0.
%\end{equation}
The previous equation can be written as
\begin{equation*}
\begin{gathered}
\int_{\eps_1}^{T-\eps_1}\int_{\mathbb{T}^d}\partial_t\theta^{\eps}
z|\theta^{\eps}|^{z-1}\sgn\theta^{\eps}\varphi
dxdt+\int_{\eps_1}^{T-\eps_1}\int_{\mathbb{T}^d}\mbox{div}_x(\theta^{\eps}
u^{\eps})z|\theta^{\eps}|^{z-1}\sgn\theta^{\eps}\varphi
dxdt\\=\int_{\eps_1}^{T-\eps_1}\int_{\mathbb{T}^d}\mbox{div}_x\left(\theta^{\eps}u^{\eps}-(\theta
u)^{\eps}\right)z|\theta^{\eps}|^{z-1}\sgn\theta^{\eps}\varphi
dxdt=:R^{\eps}.
\end{gathered}
\end{equation*}
Our goal is to prove that
\begin{equation*}
\lim_{\eps \to 0} R^{\eps}=0.
\end{equation*}
Fix $t\in (0, T)$ and consider the integral defined by
$$
R_1^{\eps}(t):=\int_{\mathbb{T}^d}\mbox{div}_x\left(\theta^{\eps}u^{\eps}-(\theta
u)^{\eps}\right)z|\theta^{\eps}|^{z-1}\sgn\theta^{\eps}\varphi dx.
$$
We can rewrite it as
\begin{equation*}
\begin{gathered}
R_1^{\eps}(t)=\int_{\mathbb{T}^d}\mbox{div}_x
\left((\theta^{\eps}(t, x)-\theta(t, x))(u^{\eps}(t, x)-u(t, x))\right)z|\theta^{\eps}|^{z-1}\sgn\theta^{\eps}\varphi(t, x)dx\\
-\int_{\mathbb{T}^d}\mbox{div}_x\bigg(\int_{[-\eps,
\eps]^d}\eta^{\eps}(\xi)[\theta(t,
x-\xi)-\theta(t,x)]\\\quad\times[u(t, x-\xi)-u(t,x)]d\xi \bigg)
z|\theta^{\eps}(t, x)|^{z-1}\sgn\theta^{\eps}(t, x)\varphi(t, x)
dx=:I(t)+J(t),
\end{gathered}
\end{equation*}
and estimate
\begin{equation*}
\begin{split}
|I(t)|&=\left|\int_{\mathbb{T}^d}\mbox{div}_x\left((\theta^{\eps}-\theta)(u^{\eps}-u)\right)z|\theta^{\eps}|^{z-1}\sgn\theta^{\eps}\varphi
dx\right|\\&=\left|\int_{\mathbb{T}^d}\left((\theta^{\eps}-\theta)(u^{\eps}-u)\right)\cdot\nabla\left(z|\theta^{\eps}|^{z-1}\sgn\theta^{\eps}\varphi\right)
dx\right|\\&\leq\left|\int_{\mathbb{T}^d}\left((\theta^{\eps}-\theta)(u^{\eps}-u)\right)\cdot\left(z|\theta^{\eps}|^{z-1}\sgn\theta^{\eps}\nabla\varphi\right)
dx\right|\\&\quad+\left|\int_{\mathbb{T}^d}\left((\theta^{\eps}-\theta)(u^{\eps}-u)\right)\cdot\nabla\left(z|\theta^{\eps}|^{z-1}\sgn\theta^{\eps}\right)\varphi
dx\right|\\&=:|I_1(t)|+|I_2(t)|.
\end{split}
\end{equation*}
We estimate $I_1(t)$ by the generalized H\"older
inequality
\begin{equation*}
\begin{split}
&|I_1(t)|\leq\left|\int_{\mathbb{T}^d}\left((\theta^{\eps}-\theta)(u^{\eps}-u)\right)\cdot\left(z|\theta^{\eps}|^{z-1}\sgn\theta^{\eps}\nabla\varphi\right)(t,x)
dx\right|\\&\leq|z|\|\varphi(t, \cdot)\|_{C^1}\|\theta^{\eps}(t,
\cdot)-\theta(t, \cdot)\|_{L^p(\mathbb{T}^d)}\|u^{\eps}(t,
\cdot)-u(t, \cdot)\|_{L^q(\mathbb{T}^d)}\left\||\theta^{\eps}(t,
\cdot)|^{\Re{z}-1}\right\|_{L^p(\mathbb{T}^d)}.
\end{split}
\end{equation*}
Now, we use the inequality
\begin{equation}\label{simbesov}
\|\theta^\eps(t, \cdot)-\theta(t, \cdot)\|_{L^p(\mathbb{T}^d)}\le
C\eps^\alpha \|\theta(t, \cdot)\|_{\dot{B}_p^{\alpha,
\infty}(\mathbb{T}^d)},
\end{equation}
which is stated e.g.\ in the paper \cite{constetiti}.
Similarly, we have
\begin{equation}\label{simbesov1}
\|u^\eps(t, \cdot)-u(t, \cdot)\|_{L^q(\mathbb{T}^d)}\le
C\eps^{\beta}\|u(t, \cdot)\|_{\dot{B}_q^{\beta, \infty}(\mathbb{T}^d)}.
\end{equation}
In view of Lemma \ref{L1}, \eqref{simbesov} and \eqref{simbesov1}, we
obtain
\begin{equation*}
\begin{split}
&|I_1(t)|\\
&\leq C|z|\|\varphi(t,
\cdot)\|_{C^1}\eps^{\alpha}\|\theta(t,
\cdot)\|_{\dot{B}_p^{\alpha,\infty}(\mathbb{T}^d)}\eps^{\beta}\|u(t,
\cdot)\|_{\dot{B}_q^{\beta,\infty}(\mathbb{T}^d)}\eps^{-\frac{d(\Re
z-2)}{p}} \|\theta(t,
\cdot)\|_{\dot{B}_p^{\alpha,\infty}(\mathbb{T}^d)}^{\Re{z}-1} \\&\leq
C|z|\|\varphi(t, \cdot)\|_{C^1}\|\theta(t,
\cdot)\|_{\dot{B}_p^{\alpha,\infty}(\mathbb{T}^d)}^{\Re z}\|u(t,
\cdot)\|_{\dot{B}_q^{\beta,\infty}(\mathbb{T}^d)}\eps^{\alpha+\beta-\frac{d(\Re
z-2)}{p}}.
\end{split}
\end{equation*}
Now, we integrate the last inequality
%\begin{equation}\nonumber
%|I_1(t)|\leq C|z|\|\varphi(t, \cdot)\|_{C^1}\|\theta(t,
%\cdot)\|_{\dot{B}_p^{\alpha,\infty}(\mathbb{T}^d)}^{\Re z}\|u(t,
%\cdot)\|_{\dot{B}_q^{\beta,\infty}(\mathbb{T}^d)}\eps^{\alpha+\beta-\frac{d(\Re
%z-2)}{p}}
%\end{equation}
over $(0, T)$. Since $\operatorname{supp}\|\varphi(t, \cdot)\|_{C^1}\subset
(\varepsilon_1, T-\varepsilon_1)$, we have
\begin{equation}\nonumber
\int_0^{T}|I_1(t)|dt\leq
C|z|\|\varphi\|_{C^1}\eps^{\alpha+\beta-\frac{d(\Re
z-2)}{p}}\int_{\varepsilon_1}^{T-\varepsilon_1}\|\theta(t,
\cdot)\|_{\dot{B}_p^{\alpha,\infty}(\mathbb{T}^d)}^{\Re z}\|u(t,
\cdot)\|_{\dot{B}_q^{\beta,\infty}(\mathbb{T}^d)}dt.
\end{equation}
Since $2<\Re(z)<2+\gamma$ and $\frac{2+\gamma}{p_1}+\frac1{p_2}\le1$,
then by H\"older's inequality we get
\begin{equation}\nonumber
\begin{split}
\int_0^T|I_1(t)|dt&\leq
C|z|\|\varphi\|_{C^1}\eps^{\alpha+\beta-\frac{d(\Re
z-2)}{p}}\|\theta\|_{L^{p_1}(\varepsilon_1, T-\varepsilon_1;\,
\dot{B}_p^{\alpha,\infty}(\mathbb{T}^d))}^{\Re
z}\\&\quad\times\|u\|_{L^{p_2}(\varepsilon_1, T-\varepsilon_1;\,
\dot{B}_q^{\beta,\infty}(\mathbb{T}^d))}.
\end{split}
\end{equation}
Now, we estimate $I_2(t)$ for fixed $t$:
\begin{equation*}
\begin{split}
|I_2(t)|&\leq|z^2-z|\|\varphi(t,
\cdot)\|_C\int_{\mathbb{T}^d}|(\theta^{\eps}-\theta)(t,
x)||(u^{\eps}-u)(t, x)||\theta^{\eps}(t, x)|^{\Re z-2}|\\&\quad\times|\nabla
\theta^{\eps}(t, x)|dx\leq|z^2-z|\|\varphi(t,
\cdot)\|_C\|(\theta^{\eps}-\theta)(t,
\cdot)\|_{L^p(\mathbb{T}^d)}\\&\quad\times\|(u^{\eps}-u)(t,
\cdot)\|_{L^q(\mathbb{T}^d)}\left\||\theta^{\eps}|^{\Re z-2}|\nabla
\theta^{\eps}|\right\|_{L^p(\mathbb{T}^d)}.
\end{split}
\end{equation*}
By H\"older's inequality we obtain
\begin{equation*}
\begin{split}
\left\||\theta^{\eps}(t, \cdot)|^{\Re z-2}|\nabla \theta^{\eps}(t,
\cdot)|\right\|&_{L^p(\mathbb{T}^d)}=\left(\int_{\mathbb{T}^d}|\theta^{\eps}(t,
x)|^{p(\Re z-2)}|\nabla \theta^{\eps}(t,
x)|^pdxdt\right)^{\frac{1}{p}}\leq\\&
\left(\left\|\left|\theta^{\eps}(t, \cdot)\right|^{p(\Re
z-2)}\right\|_{L^\frac{\Re z-1}{\Re
z-2}(\mathbb{T}^d)}\left\|\left|\nabla\theta^{\eps}(t,
\cdot)\right|^p\right\|_{L^{\Re
z-1}(\mathbb{T}^d)}\right)^\frac1{p}\\&=\left\|\theta^{\eps}(t,
\cdot)\right\|_{L^{p(\Re z-1)}(\mathbb{T}^d)}^{\Re
z-2}\left\|\nabla\theta^{\eps}(t, \cdot)\right\|_{L^{p(\Re
z-1)}(\mathbb{T}^d)}.
\end{split}
\end{equation*}
We take $s:=\Re z$ in Lemma \ref{estfornab} to deduce
\begin{equation}\label{ine2}
\left\|\theta^{\eps}(t, \cdot)\right\|_{L^{p(\Re
z-1)}(\mathbb{T}^d)}^{\Re{z}-2}\leq C(p)\eps^{-\frac{d(\Re
z-2)^2}{p(\Re z-1)}}\|\theta(t,
\cdot)\|_{L^p(\mathbb{T}^d)}^{\Re{z}-2}.
\end{equation}
Combining \eqref{estnabth} and \eqref{ine2}, we obtain
\begin{equation*}
\left\||\theta^{\eps}(t, \cdot)|^{\Re z-2}|\nabla \theta^{\eps}(t,
\cdot)|\right\|_{L^p(\mathbb{T}^d)}\leq
C(p)\eps^{\alpha-1-\frac{d(\Re z-2)}{p}}\|\theta(t,
\cdot)\|_{\dot{B}_p^{\alpha,\infty}(\mathbb{T}^d)}^{\Re z-1}.
\end{equation*}
Thus, combining all obtained estimates we get
\begin{equation}\nonumber
\begin{split}
|I_2(t)|&\leq
C|z||z-1|\|\varphi\|_{C}\eps^{\alpha}\eps^{\beta}\|u(t,
\cdot)\|_{\dot{B}_q^{\beta,\infty}(\mathbb{T}^d)}\eps^{\alpha-1-\frac{d(\Re
z-2)}{p}}\|\theta(t,
\cdot)\|_{\dot{B}_p^{\alpha,\infty}(\mathbb{T}^d)}^{\Re{z}}\\& =
C|z||z-1|\|\varphi\|_{C}\|\theta(t,
\cdot)\|_{\dot{B}_p^{\alpha,\infty}(\mathbb{T}^d)}^{\Re z}\|u(t,
\cdot)\|_{\dot{B}_q^{\beta,\infty}(\mathbb{T}^d)}\eps^{2\alpha+\beta-1-\frac{d(\Re
z-2)}{p}}.
\end{split}
\end{equation}
Finally, integrating the last inequality
%\begin{equation}\nonumber
%|I_2(t)|\leq C|z||z-1|\|\varphi(t, \cdot)\|_{C}\|\theta(t,
%\cdot)\|_{\dot{B}_p^{\alpha,\infty}(\mathbb{T}^d)}^{\Re z}\|u(t,
%\cdot)\|_{\dot{B}_q^{\beta,\infty}(\mathbb{T}^d)}\eps^{2\alpha+\beta-1-\frac{d(\Re
%z-2)}{p}}.
%\end{equation}
over $(0, T)$ similarly to the estimate for $I_1$, we arrive at
\begin{equation}\nonumber
\begin{split}
\int_0^T|I_2(t)|dt&\leq
C|z||z-1|\|\varphi\|_{C}\eps^{2\alpha+\beta-1-\frac{d(\Re
z-2)}{p}}\|\theta\|_{L^{p_1}(\varepsilon_1, T-\varepsilon_1;\,
\dot{B}_p^{\alpha,\infty}(\mathbb{T}^d))}^{\Re
z}\\&\quad\times\|u\|_{L^{p_2}(\varepsilon_1, T-\varepsilon_1;\,
\dot{B}_q^{\beta,\infty}(\mathbb{T}^d))}.
\end{split}
\end{equation}
Hence we obtain the estimate for $I(t)$:
\begin{equation}\label{firstteri}
\begin{split}
\int_0^T|I(t)|dt&\leq
C(z)\|\varphi\|_{C^1}\eps^{2\alpha+\beta-1-\frac{d(\Re
z-2)}{p}}\\&\quad\times\|\theta\|_{L^{p_1}([\varepsilon_1, T-\varepsilon_1],\,
\dot{B}_p^{\alpha,\infty}(\mathbb{T}^d))}^{\Re
z}\|u\|_{L^{p_2}([\varepsilon_1, T-\varepsilon_1],\,
\dot{B}_q^{\beta,\infty}(\mathbb{T}^d))}.
\end{split}
\end{equation}
Now, we consider an estimate for the second term $J(t)$:
\begin{equation*}
\begin{split}
|J(t)|&=\bigg|\int_{\mathbb{T}^d}\mbox{div}_x\bigg(\int_{[-\eps,
\eps]^d}\eta^{\eps}(\xi)[\theta(t, x-\xi)-\theta(t,x)]\\&\quad
\times[u(t, x-\xi)-u(t,x)]d\xi\bigg) z|\theta^{\eps}(t,
x)|^{z-1}\sgn\theta^{\eps}(t, x)\varphi(t, x) dx\bigg|\\&=
\bigg|\int_{\mathbb{T}^d}\int_{[-\eps,
\eps]^d}\eta^{\eps}(\xi)[\theta(t, x-\xi)-\theta(t,x)]\\&\quad\times
[u(t, x-\xi)-u(t,x)]d\xi \cdot \nabla_x\left(z|\theta^{\eps}(t,
x)|^{z-1}\sgn\theta^{\eps}(t, x)\varphi(t, x)\right) dx\bigg|\\&
\leq\bigg|\int_{\mathbb{T}^d}\int_{[-\eps,
\eps]^d}\eta^{\eps}(\xi)[\theta(t, x-\xi)-\theta(t,x)]\\&\quad\times
[u(t, x-\xi)-u(t,x)]d\xi  \cdot
\left(z|\theta^{\eps}|^{z-1}\sgn\theta^{\eps}\nabla_x\varphi\right)
dx\bigg|\\&\quad+\bigg|\int_{\mathbb{T}^d}\int_{[-\eps,
\eps]^d}\eta^{\eps}(\xi)[\theta(t, x-\xi)-\theta(t,x)]\\&\quad\times
[u(t, x-\xi)-u(t,x)]d\xi \cdot
\left(\nabla_xz|\theta^{\eps}|^{z-1}\sgn\theta^{\eps}\right)\varphi
dx\bigg|\\&=:|J_1(t)|+|J_2(t)|.
\end{split}
\end{equation*}
We estimate $J_1(t)$ by using the Cauchy-Schwarz  and
generalized H\"older inequalities:
\begin{equation*}
\begin{split}
|J_1(t)|&\leq |z|\|\varphi(t,
\cdot)\|_{C^1}\int_{\mathbb{T}^d}\Big|\int_{[-\eps,
\eps]^d}\eta^{\eps}(\xi)[\theta(t, x-\xi)-\theta(t,x)]\\&
\quad\times [u(t, x-\xi)-u(t,x)]d\xi\Big| |\theta^{\eps}(x,
t)|^{z-1}dx\\&\leq|z| \|\varphi(t,
\cdot)\|_{C^1}\int_{\mathbb{T}^d}\int_{[-\eps,
 \eps]^d}\eta^{\eps}(\xi)|\theta(t,
x-\xi)-\theta(t,x)|\\& \quad\times |u(t, x-\xi)-u(t,x)|d\xi
|\theta^{\eps}(x, t)|^{\Re z-1}dx\\& \leq|z|\|\varphi(t,
\cdot)\|_{C^1}\int_{\mathbb{T}^d}\eta^{\eps}(\xi)\|\theta(t,
\cdot-\xi)-\theta(t,\cdot)\|_{L^p(\mathbb{T}^d)}\\& \quad\times
\|u(t, \cdot-\xi)-u(t,\cdot)\|_{L^q(\mathbb{T}^d)}
\left\||\theta^{\eps}(t, \cdot)|^{\Re
z-1}\right\|_{L^p(\mathbb{T}^d)}d\xi.
\end{split}
\end{equation*}
Since $\theta\in \dot{B}_p^{\alpha, \infty}$, $u\in \dot{B}_q^{\beta, \infty}$ for a.e.\ $t\in [0, T]$,
we have
\begin{equation*}
\|\theta(t, \cdot-\xi)-\theta(t,\cdot)\|_{L^p(\mathbb{T}^d)}\leq
|\xi|^\alpha\|\theta(t, \cdot)\|_{\dot{B}_p^{\alpha,
\infty}(\mathbb{T}^d)}
\end{equation*}
and
\begin{equation*}
\|u(t,
\cdot-\xi)-u(t,\cdot)\|_{L^q(\mathbb{T}^d)}\leq|\xi|^\beta\|u(t,
\cdot)\|_{\dot{B}_p^{\beta, \infty}(\mathbb{T}^d)}.
\end{equation*}
Observe that
\begin{equation*}
\int_{[-\eps, \eps]^d} \eta^{\eps}(\xi)|\xi|^{\alpha+\beta}d\xi=C
\eps^{\alpha+\beta}
\end{equation*}
as $\operatorname{supp}\eta^{\eps}\subset\{\xi\in \mathbb{R}^d: |\xi|<\eps\}$. Hence,
we have
\begin{equation*}
|J_1(t)|\leq C|z|\|\varphi(t, \cdot)\|_{C^1}\|\theta\|_{\dot{B}_p^{\alpha,
\infty}}^{\Re z}\|u\|_{\dot{B}_q^{\beta,
\infty}}\eps^{\alpha+\beta-\frac{d(\Re z-2)}{p}}.
\end{equation*}
Then we integrate the last inequality over $[0, T]$:
$$
\int_0^T|J_1(t)|dt\leq
C|z|\|\varphi\|_{C^1}\eps^{\alpha+\beta-\frac{d(\Re
z-2)}{p}}\int_{\varepsilon_1}^{T-\varepsilon_1}\|\theta(t,
\cdot)\|_{\dot{B}_p^{\alpha,\infty}(\mathbb{T}^d)}^{\Re z}\|u(t,
\cdot)\|_{\dot{B}_q^{\beta, \infty}(\mathbb{T}^d)}dt.
$$
The last integral converges.
Similarly, we estimate $J_2(t)$:
\begin{equation*}
\begin{split}
|J_2(t)|&=\bigg|\int_{\mathbb{T}^d}\int_{[-\eps,
\eps]^d}\eta^{\eps}(\xi)[\theta(t, x-\xi)-\theta(t,x)]\\&\quad\times
[u(t, x-\xi)-u(t,x)]d\xi  \cdot
\left(z(z-1)|\theta^{\eps}|^{z-2}\nabla_x \theta^{\eps}(t,
x)\right)\varphi dx\bigg|\\&
\leq|z||z-1|\|\varphi\|_{C}\int_{\mathbb{T}^d}\Big|\int_{[-\eps,
\eps]^d}\eta^{\eps}(\xi)[\theta(t, x-\xi)-\theta(t,x)]\\&
\quad\times [u(t, x-\xi)-u(t,x)]d\xi \cdot \nabla\theta^{\eps}(t,
x)|\theta^{\eps}(t, x)|^{z-2}\Big|dx\\&\leq|z||z-1|
\|\varphi\|_{C}\int_{\mathbb{T}^d}\int_{[-\eps,
\eps]^d}\eta^{\eps}(\xi)|\theta(t, x-\xi)-\theta(t,x)|\\&
\quad\times |u(t, x-\xi)-u(t,x)|d\xi |\theta^{\eps}(t, x)|^{\Re
z-2}|\nabla \theta^{\eps}(t, x)|dx\\&
\leq|z||z-1|\|\varphi\|_{C}\int_{[-\eps,
\eps]^d}\eta^{\eps}(\xi)\|\theta(t,
\cdot-\xi)-\theta(t,\cdot)\|_{L^p(\mathbb{T}^d)}\\& \quad\times
\|u(t, \cdot-\xi)-u(t,\cdot)\|_{L^q(\mathbb{T}^d)}
\left\||\theta^{\eps}(t, \cdot)|^{\Re z-2}|\nabla\theta^{\eps}(t,
\cdot)|\right\|_{L^p(\mathbb{T}^d)}d\xi.
\end{split}
\end{equation*}
Using the previous inequalities and Lemma \ref{estfornab}, we obtain
\begin{equation}\nonumber
|J_2(t)|\leq C|z||z-1|\|\varphi\|_{C}\left\|\theta(t,
\cdot)\right\|_{\dot{B}_p^{\alpha,\infty}}^{\Re z}\|u(t,
\cdot)\|_{\dot{B}_q^{\beta, \infty}}\eps^{2\alpha+\beta-1-\frac{d(\Re
z-2)}{p}}.
\end{equation}
Hence,
\begin{equation}\nonumber
\begin{split}
\int_0^T|J_2(t)|dt&\leq
C|z||z-1|\|\varphi\|_{C}\left\|\theta\right\|_{L^{p_1}([\eps_1, T-\eps_1],
\dot{B}_p^{\alpha,\infty}(\mathbb{T}^d))}^{\Re z}\\&\quad\times\|u\|_{L^{p_2}([\eps_1, T-\eps_1],
\dot{B}_q^{\beta, \infty}(\mathbb{T}^d))}\eps^{2\alpha+\beta-1-\frac{d(\Re
z-2)}{p}}.
\end{split}
\end{equation}
Thus
\begin{equation}\nonumber
\begin{split}
\int_0^T|J(t)|dt&\leq
C(z)\|\varphi\|_{C^1}\left\|\theta\right\|_{L^{p_1}([\eps_1, T-\eps_1],
\dot{B}_p^{\alpha,\infty}(\mathbb{T}^d))}^{\Re z}\\&\quad\times\|u\|_{L^{p_2}([\eps_1, T-\eps_1],
\dot{B}_q^{\beta, \infty}(\mathbb{T}^d))}\eps^{2\alpha+\beta-1-\frac{d(\Re
z-2)}{p}}.
\end{split}
\end{equation}
In other words,
\begin{equation}\nonumber
\begin{split}
&\left|\int_0^T\int_{\mathbb{T}^d}\partial_t\theta^{\eps}
z|\theta^{\eps}|^{z-1}\sgn\theta^{\eps}\varphi
-\mbox{div}_x(\theta^{\eps}
u^{\eps})z|\theta^{\eps}|^{z-1}\sgn\theta^{\eps}\varphi dxdt\right|\\ 
\le& C(z)\|\varphi\|_{C^1}\left\|\theta\right\|_{L^{p_1}([\eps_1, T-\eps_1],
\dot{B}_p^{\alpha,\infty}(\mathbb{T}^d))}^{\Re z}\|u\|_{L^{p_2}([\eps_1, T-\eps_1],
\dot{B}_q^{\beta, \infty}(\mathbb{T}^d))}\eps^{2\alpha+\beta-1-\frac{d(\Re
z-2)}{p}}.
\end{split}
\end{equation}
In particular, since $\Re z< 2+\gamma$, then
\begin{equation}\nonumber
\int_0^T\int_{\mathbb{T}^d}\partial_t\theta^{\eps}
z|\theta^{\eps}|^{z-1}\sgn\theta^{\eps}\varphi
-\mbox{div}_x(\theta^{\eps}
u^{\eps})z|\theta^{\eps}|^{z-1}\sgn\theta^{\eps}\varphi dxdt=o(1)
\end{equation}
as $\eps\to+0$, whenever $2\alpha+\beta>1+\frac{d\gamma}{p}$.
Since $\theta^{\eps}(t, \cdot), \, u^{\eps}(t, \cdot)\in
C^\infty(\mathbb{T}^d)$ for a.e.\ $t\in [0, T]$ and $\mbox{div}_x(u^{\eps}(t, \cdot))=0$, we have
$$
\mbox{div}_x(\theta^{\eps} u^{\eps})z|\theta^{\eps}|^{z-1}\sgn
(\theta^{\eps})=\mbox{div}_x(|\theta^{\eps}|^z u^{\eps}).
$$
Consequently, integration by parts yields
\begin{equation*}
\begin{split}
\int_{\mathbb{T}^d}\mbox{div}_x(\theta^{\eps}
u^{\eps})z|\theta^{\eps}|^{z-1}\sgn\theta^{\eps}\varphi
dx& =\int_{\mathbb{T}^d}\mbox{div}_x(|\theta^{\eps}|^z
u^{\eps})\varphi dx\\&=-\int_{\mathbb{T}^d}|\theta^{\eps}|^z
u^{\eps}\cdot\nabla \varphi dx.
\end{split}
\end{equation*}
On the other hand, by the chain rule it holds that
\begin{equation*}
\int_0^T\int_{\mathbb{T}^d}\partial_t\theta^{\eps}
z|\theta^{\eps}|^{z-1}\sgn\theta^{\eps}\varphi
dxdt-\int_0^T\int_{\mathbb{T}^d}\partial_t|\theta^{\eps}|^z\varphi
dxdt=0.
\end{equation*}
Therefore,
\begin{equation*}
\begin{gathered}\label{weakform}
\int_0^T\int_{\mathbb{T}^d}\partial_t\theta^{\eps}
z|\theta^{\eps}|^{z-1}\sgn(\theta^{\eps})\varphi dxdt+
\int_0^T\int_{\mathbb{T}^d}\mbox{div}_x(\theta^{\eps} u^{\eps})z|\theta^{\eps}|^{z-1}\sgn\theta^{\eps}\varphi dxdt\\
=\int_0^T\int_{\mathbb{T}^d}\partial_t |\theta^{\eps}|^{z}\varphi dxdt+\int_0^T\int_{\mathbb{T}^d}\mbox{div}_x(|\theta^{\eps}|^z u^{\eps})\varphi dxdt\\
=-\int_0^T\int_{\mathbb{T}^d} |\theta^{\eps}|^{z}\partial_t\varphi dxdt-\int_0^T\int_{\mathbb{T}^d}|\theta^{\eps}|^zu^{\eps}\cdot\nabla\varphi dxdt\\
=-\int_0^T\int_{\mathbb{T}^d}
|\theta^{\eps}|^{z}\left(\partial_t\varphi+u^\eps\cdot\nabla\varphi\right)
dxdt=:F_{\eps}(z).
\end{gathered}
\end{equation*}
From the above estimates that we obtained, we have
\begin{equation*}
\lim_{\eps\to 0}F_\eps(z)=0.
\end{equation*}
On the other hand, as $\eps\to 0$,
$|\theta^{\eps}|^{z}\to|\theta|^{z}$ and $u^\eps \to u$, and thus
\begin{equation*}
\lim_{\eps\to 0}F_\eps(z)=F(z),
\end{equation*}
where $$F(z)=-\int_{0}^{T}\int_{\mathbb{T}^d}
|\theta|^{z}\left(\partial_t\varphi+u\cdot\nabla\varphi\right)
dxdt.$$
\end{proof}

Note that if we set the test function
$\varphi(t, x)=\phi(t)$ to depend only on the variable $t$, then we
obtain the following corollary on the \emph{global} conservation of $L^p$ norms:

\begin{corollary}\label{cor:con.en}
Let a pair $\theta$ be a weak solution of the transport problem
\eqref{transport} on $(0, T)\times \mathbb{T}^d$. Assume that
\begin{equation*}
\theta\in \dot{B}_p^{\alpha, \infty},\quad u\in \dot{B}_q^{\beta,
\infty}\quad\text{for some}\quad 1\leq p,q\leq\infty\quad\text{and}\quad 0<\alpha,\beta<1
\end{equation*}  such that $\frac{2}{p}+\frac{1}{q}=1$. Let $\gamma>0$ and assume $2\alpha+\beta>1+\frac{d\gamma}{p}$. Then, for any $z\in\{z\in\mathbb{C}:\quad 2<\Re z<2+\delta\}$,
the total conservation of $L^z$ norm of $\theta$ is valid, that
is,
\begin{equation}\label{Th1}
\int_{0}^{T}\phi'(t)\int_{\mathbb{T}^d} |\theta|^{z}dxdt=0.
\end{equation} 
\end{corollary}

\section{Conclusion of main results}

Using Lemma \ref{analyticity1}, we know that $$F(z):=\int_{0}^{T}\int_{{\mathbb{T}}^d}(\partial_t\varphi(t)+u\cdot\nabla\varphi)|\theta|^zdxdt$$ is an analytic function on $1<\Re z<r$, where $r$ is any exponent for which $|\theta|^ru$ is locally integrable in $(0,T)\times\mathbb{T}^d$. By Theorem \ref{Thmain}, $F=0$ on an open set, so that Theorem~\ref{mainth} follows from the unique continuation principle for analytic functions.

 Theorem~\ref{mainthactive} then immediately follows by Proposition \ref{boundedness}.


\begin{thebibliography}{99}
	
\bibitem{aizenman}	
M. Aizenman: \textit{A sufficient condition for the avoidance of sets by measure preserving
flows in $\mathbb{R}^n$}. Duke Math. J. 45, 1978, 809--814.

\bibitem{abc1}
G. Alberti, S. Bianchini, G. Crippa. \textit{Structure of level sets and Sard-type properties of Lip-schitz maps}. Ann. Scuola Norm. Sup. Pisa Cl. Sci., 12, 2013, 863--902.

\bibitem{abc2}
G. Alberti, S. Bianchini, G. Crippa: \textit{A uniqueness result for the continuity equation in two dimensions}. J. Eur. Math. Soc. (JEMS), 16, 2014.

\bibitem{Bahori}
H. Bahouri, J.-Y. Chemin, R. Danchin:
\textit{Fourier Analysis and Nonlinear Partial Differential Equations}. In Grundlehren der mathematischen Wissenschaften, 343, Springer, 2011.

\bibitem{BT18}
C. Bardos and E. S. Titi: \textit{Onsager's Conjecture for the incompressible Euler equations
in bounded domains}. Arch. Ration. Mech. Anal., 228(1), 2018, 197--207.

\bibitem{BTW18}
C. Bardos, E. S. Titi, E. Wiedemann: \textit{Onsager's conjecture with physical boundaries and an application to the vanishing viscosity limit}. Preprint, 2018, arXiv:1803.04939. 

%\bibitem{Bear}
%J. Bear: \textit{Dynamics of Fluids in Porous Media}. Courier Dover Publications, New York, 1972.

\bibitem{BSV}
T. Buckmaster, S. Shkoller, and V. Vicol: \textit{Nonuniqueness of weak solutions to the SQG equation}. Preprint, 2016, arXiv:1610.00676.

\bibitem{CCFS08}
A. Cheskidov, P. Constantin, S. Friedlander,  R. Shvydkoy: \textit{Energy conservation and Onsager's conjecture for the Euler equations}. Nonlinearity 21 (6), 2008, 1233--1252. 

\bibitem{colluorau}
F. Colombini, T. Luo and J. Rauch: \textit{Uniqueness and nonuniqueness for nonsmooth divergence free transport}. Seminaire: \'Equations aux D\'eriv\'ees Partielles, 2002–2003, S\'emin. \'Equ. D\'eriv. Partielles, Exp. No. XXII, \'Ecole Polytech., Palaiseau, 1–21, 2003.

\bibitem{constetiti}
P.~Constantin, W.~E, and E.~S. Titi:
\textit{Onsager's conjecture on the energy conservation for solutions of
{E}uler's equation}. Comm. Math. Phys., 165(1), 1994, 207--209.

%\bibitem{Const}
%P. Constantin,  A.J.Majda,  E. Tabak: \textit{Formation of strong fronts in the 2-D quasigeostrophic thermal active scalar}. %Nonlinearity 7(6), 1994, 1495--1533.

\bibitem{CFG}
D. C\'{o}rdoba, D. Faraco, F. Gancedo: \textit{Lack of uniqueness for weak solutions of the incompressible porous media equation}. Arch. Ration. Mech. Anal., 200, 2011, 725--746.

%\bibitem{Orive}
%D. C\'{o}rdoba, F. Gancedo, R. Orive: \textit{Analytical behavior of two-dimensional incompressible flow in porous media}. J. Math. %Phys. 48(6), 065206, 2007.

\bibitem{cgsw1}
G. Crippa, N. Gusev, S. Spirito, E. Wiedemann: \textit{Failure of the chain rule for the divergence of bounded vector fields}. Ann. Sc. Norm. Super. Pisa Cl. Sci. 17, 2017, 1--18.

\bibitem{cgsw2}
G. Crippa, N. Gusev, S. Spirito, E. Wiedemann: \textit{Non-uniqueness and prescribed energy for the continuity equation}. Commun. Math. Sci. 13, 2015, 1937--1947.

\bibitem{depauw}
N. Depauw: \textit{Non unicit\'{e} des solutions born\'{e}es pour un champ de vecteurs BV en dehors d'un hyperplan}. C.R. Math. Acad. Sci. Paris 337(4), 2003, 249--252.



\bibitem{dipernalions}
R. DiPerna and P.-L. Lions: \textit{Ordinary differential equations, transport theory and Sobolev spaces}. Invent. Math. 98, 1989, 511--547.

\bibitem{dipernalionsboltzmann}
R. DiPerna and P.-L. Lions:\textit{On the Cauchy problem for Boltzmann equations: global existence and weak stability}. Ann. of Math. (2), 130, 1989, 321--366.

\bibitem{DE17}
T. D. Drivas and G. L. Eyink: \textit{An Onsager singularity theorem for turbulent solutions of compressible Euler equations}. Commun. Math. Phys., DOI: 10.1007/s00220-017-3078-4, 2018.

\bibitem{DN18}
T. D. Drivas and  H. Q. Nguyen: \textit{Onsager’s conjecture and anomalous dissipation on domains with boundary}. Preprint 2018, arXiv: 1803.05416v2.

\bibitem{duchonrobert}
J. Duchon and R. Robert: \textit{Inertial energy dissipation for weak solutions of incompressible Euler and Navier-Stokes equations}. Nonlinearity 13, 2000, 249--255.

\bibitem{eyink94}
G. L. Eyink: \textit{Energy dissipation without viscosity in ideal hydrodynamics}. I. Fourier analysis and local energy transfer, Phys. D 78, 1994, 222--240.

\bibitem{eyink01}
G. L. Eyink: \textit{Dissipation in turbulent solutions of $2D$ Euler equations}. Nonlinearity 14 (4), 2001, 787--802.  

\bibitem{FGGW}
E. Feireisl, P. Gwiazda, A. \'{S}wierczewska-Gwiazda, E. Wiedemann: \textit{Regularity and energy conservation for the compressible Euler equations}. Arch. Ration. Mech. Anal. 223, 2017, 1375--1395.

\bibitem{feireislcompact}
E. Feireisl: \textit{On compactness of solutions to the compressible isentropic Navier-Stokes equations when the density is not square integrable}. Comment. Math. Univ. Carol. 42, 2001, 83--98.

\bibitem{FW17}
U. S. Fjordholm and E. Wiedemann: \textit{Statistical solutions and Onsager's conjecture}. Phys. D, (to appear). Preprint 2017, arXiv:1706.04113.

\bibitem{GRA1}
L. Grafakos: \textit{Classical Fourier Analysis}.
Third Edition, Graduate Texts in Math., no 249,
Springer, New York, 2014.

\bibitem{GRA2}
L. Grafakos: \textit{Modern Fourier Analysis}.
Third Edition, Graduate Texts in Math., no 250,
Springer, New York, 2014.

\bibitem{Martin}
P. Gwiazda, M. Mich\'{a}lek, A. \'{S}wierczewska-Gwiazda: \textit{A note on weak solutions of conservation laws and energy/entropy conservation}. Preprint 2017, arXiv:1706.10154.

%\bibitem{Held}
%I.M. Held,  R.T. Pierrehumbert,  S.T. Garner,  K.L. Swanson: \textit{Surface quasi-geostrophic dynamics}. J. Fluid Mech. 282, 1995, %1--20.

\bibitem{PI}
P. Isett, V. Vicol: \textit{H\"{o}lder continuous solutions of active scalar equations}. Ann. PDE, 1, 2015, 1--77.

\bibitem{LSh}
T. M. Leslie and R. Shvydkoy: \textit{The energy balance relation for weak solutions of the density-dependent Navier-Stokes equations}. J. Differential Equations 261, 2016, 3719--3733.

\bibitem{lionscomp}
P.-L. Lions: \textit{Mathematical Topics in Fluid Mechanics}. Vol. 2, Compressible Models, Clarendon Press, Oxford Science Publications, Oxford 1998.

\bibitem{lopes06}
M. Lopes Filho, A. Mazzucato, H. Nussenzveig Lopes: \textit{Weak solutions, renormalized solutions and enstrophy defects in $2D$ turbulence}. Arch. Ration. Mech. Anal., 179, 2006, 353--387.

\bibitem{RRS16}
J. C. Robinson, J. L. Rodrigo, J. W. D. Skipper: \textit{Energy conservation in the $3D$
Euler equations on $\mathbb{T}^2\times\mathbb{R}^{+}$}. Preprint 2017, arXiv:1611.00181.

\bibitem{RRS18}
J. C. Robinson, J. L. Rodrigo, J. W. D. Skipper: \textit{Energy conservation in the $3D$
Euler equations on $\mathbb{T}^2\times\mathbb{R}^{+}$ for weak solutions defined without reference to the pressure}. Asymptot. Anal., 2018, (to appear).

\bibitem{scheffer} 
V. Scheffer: \textit{An inviscid flow with compact support in space-time}. J. of Geom. Anal., 3, 1993, 343--401.



\bibitem{shnirelman}
A. Shnirelman: \textit{On the nonuniqueness of weak solutions of the Euler equations}. Comm. Pure Appl. Math., 50, 1997, 1261--1286.

\bibitem{RSh}
R. Shvydkoy: \textit{Convex integration for a class of active scalar equations}.  J. Amer. Math. Soc., 24(4), 2011, 1159--1174.

%\bibitem{Sogge}
%C.~Sogge:
%\textit{Fourier integrals in classical analysis}. Cambridge University Press, 1993.

\bibitem{Sz12}
L.~Sz\'ekelyhidi, Jr.: \textit{Relaxation of the incompressible porous media equation.} Ann. Sci. \'Ec. Norm. Sup\'er., 45, 2012, 491--509.

\bibitem{CY17}
C. Yu: \textit{Energy conservation for the weak solutions of the compressible Navier–Stokes equations}. Arch. Ration. Mech. Anal., 225(3), 2017, 1073--1087.

\end{thebibliography}
\end{document}